\documentclass[a4paper,11pt,oneside,reqno]{amsart}
\usepackage{amsmath,amsthm, amssymb,  mathrsfs, graphicx,color}

\usepackage[T1]{fontenc}
\usepackage{lmodern}

\newtheorem{Theorem}{Theorem}
\newtheorem{Corollary}{Corollary}

\newtheorem{Proposition}{Proposition}
\newtheorem{Condition}{Condition}

\theoremstyle{remark}
\newtheorem{remark}{Remark}

\theoremstyle{definition}

\makeatletter
\def\@secnumfont{\bfseries}
\def\section{\@startsection{section}{1}\z@{.7\linespacing\@plus\linespacing}
{.5\linespacing}{\large\bfseries\centering}}

\def\@settitle{\noindent\center\parbox{.83333\textwidth}
{\baselineskip143\p@\relax\center\bfseries\Large\@title}}
\def\@setauthors{
\begingroup
  \trivlist\@topsep 3ex
  	\item\relax\center\author@andify\authors{\authors}%
  \endtrivlist
\endgroup\vspace{1ex}
}

\renewcommand{\abstract}[1]{{\center\parbox[t]{.89\textwidth}{\small{\bfseries \abstractname:\ }#1}\\}\vspace{3ex}}
\makeatother

\newcommand{\coco}[1]{\mbox{\rm(#1)}}

\newcommand{\nline}{\\&&\quad}
\newcommand{\itc}[2]{\item[{\rm(\ref{#1}$_#2$)}]\refstepcounter{enumi}}
\newcommand{\cd}[1]{Condition \coco{\ref{#1}}}

\newcommand{\CD}[2]{\coco{\ref{#1}$_{\ref{#2}}$}}

\newcommand{\fig}[2]{
\begin{center}
\includegraphics[width=#1\textwidth]{#2}
\end{center}
}

\newcommand{\figc}[3]{
\begin{center}
\includegraphics[width=#1\textwidth]{#2}
{{\it\small Fig.: #3}}\\
\end{center}
}

\newcommand{\fac}[1]{#1^+}

\newcommand{\kla}[1]{(#1)}
\newcommand{\bkla}[1]{\big(#1\big)}
\newcommand{\Bkla}[1]{\Big(#1\Big)}
\newcommand{\beck}[1]{\big[#1\big]}

\newcommand{\nor}[1]{\lVert#1\rVert}

\newcommand{\PPPP}{P}
\newcommand{\PPKK}[1]{(#1)}
\newcommand{\EEEE}{E}
\newcommand{\EEKK}[1]{(#1)}

\newcommand{\PP}[1]{\PPPP\PPKK{#1}}
\newcommand{\PPP}[1]{\PPPP_{#1}}
\newcommand{\PX}[2]{\PPP{#1}\PPKK{#2}}

\newcommand{\EE}[1]{\EEEE\EEKK{#1}}
\newcommand{\EEE}[1]{\EEEE_{#1}}
\newcommand{\EX}[2]{\EEE{#1}\EEKK{#2}}

\newcommand{\dis}[2]{\partial_{#1}#2}

\newcommand{\abscon}{\ll}
\newcommand{\radon}[2]{\frac{d#1}{d#2}}
\newcommand{\radons}[2]{d#1/d#2}
\newcommand{\res}[2]{{#1}\vert_{#2}}

\newcommand{\ad}[1]{#1^\star}

\newcommand{\oper}[1]{{\mathcal{#1}}}
\newcommand{\opset}[1]{{\mathscr #1}}

\newcommand{\borel}[1]{{\mathscr B}(#1)}

\newcommand{\gen}{{\oper A}}
\newcommand{\xen }{{\oper X}}

\newcommand{\qen}{{\oper Q}}

\newcommand{\ten}[1]{\widehat #1}

\newcommand{\mls}[1]{\mu_{#1}}
\newcommand{\ml}[2]{\mls{#1}(#2)}

\newcommand{\adml}[2]{\ad \mu_{#1}(#2)}

\newcommand{\KK}[2]{\theta(#1,#2)}
\newcommand{\mlm}[2]{\mu_{#1}^{\prime}(#2)}
\newcommand{\admlm}[2]{\mu_{#1}^{\star\prime}(#2)}
\newcommand{\MM}[2]{\overline\theta(#1,#2)}

\newcommand{\xixs}{\xi}
\newcommand{\xix}{\xixs}

\newcommand{\AFT}{Q}
\newcommand{\BEF}{W}
\newcommand{\IMB}{Z}

\newcommand{\stas}{\pi}
\newcommand{\sta}{\stas}

\newcommand{\stbef}{\stas_{\BEF}}
\newcommand{\staft}{\stas_{\AFT}}

\newcommand{\cv}[2]{\Phi_{#1}^{#2}}
\newcommand{\cvv}[1]{\Phi_{#1}^+}

\newcommand{\adsta}{\ad\stas}

\newcommand{\tact}{\tau}
\newcommand{\tim}[2]{\tau(#1,#2)}
\newcommand{\adtim}[2]{\ad\tau(#1,#2)}

\newcommand{\la}{\lambda}

\newcommand{\La}[2]{\Lambda_{#2}(#1)}

\newcommand{\adla}{\lambda^{\!\star}}

\newcommand{\adXX}[1]{\ad X_t}
\newcommand{\ind}[1]{1_{\{#1\}}}

\newcommand{\leb}{\ell}

\newcommand{\veps}{\varepsilon}

\newcommand{\hh}{\varphi}

\newcommand{\boundary}{{\partial\sE}}

\newcommand{\ace}{{\Gamma}}
\newcommand{\act}{\ace}
\newcommand{\acm}{\ad\ace}

\newcommand{\sts}{{\mathbb E}}
\newcommand{\sE}{\sts}
\newcommand{\sEE}{\sts^\circ}
\newcommand{\sEA}{\ad\sE}
\newcommand{\sEI}{\sts_{\IMB}}

\newcommand{\dom}[1]{{\opset D}(#1)}

\newcommand{\mes}{{\opset M}}
\newcommand{\mabs}{{\opset M}'}
\newcommand{\bacon}{{\opset M}'_{b}}

\newcommand{\sN}{{\mathbb N}}
\newcommand{\sR}{{\mathbb R}}

\newcommand{\por}[3]{p(#1;#2,#3)}
\newcommand{\adpor}[3]{\ad p(#1;#2,#3)}

\newcommand{\pdmp}{PDMP}

\newcommand{\seqi}[2]{(#1_{#2})_{#2=1,2,\ldots}}
\newcommand{\seq}[2]{(#1_{#2})_{#2\geq 0}}
\newcommand{\pd}[1]{\frac{\partial}{\partial #1}}

\newcommand{\afrac}[2]{#1/#2}

\newcounter{modd}
\renewcommand{\themodd}{\Roman{modd}}
\newcounter{submodd}
\renewcommand{\thesubmodd}{\Roman{modd}\alph{submodd}}

\newenvironment{vemodd}{\setcounter{submodd}{0}\vspace{3mm}\par\noindent\refstepcounter{modd}{\bf Model \themodd\ }}{}
\newenvironment{vemoddn}[1]{\setcounter{submodd}{0}\vspace{3mm}\par\noindent\refstepcounter{modd}{\bf Model \themodd}\;(#1)}{}
\newenvironment{vemoddd}{\vspace{3mm}\par\noindent\refstepcounter{submodd}{\bf Model \thesubmodd\ }}{}

\begin{document}

\title[Time Reversal of PDMPs]{On Time Reversal of Piecewise Deterministic Markov Processes}
\author{Andreas L\"opker}
\address{Helmut Schmidt University Hamburg, Postfach 700822, 22008 Hamburg}
\email{lopker@hsu-hh.de}
\author{Zbigniew Palmowski}
\address{University of Wroc\l aw, pl.\ Grunwaldzki 2/4, 50-384 Wroc\l aw, Poland}
\email{zpalma@math.uni.wroc.pl}
\date{\today}
\maketitle

\abstract{We study the time reversal of a general \pdmp. The time reversed process is defined as $X_{(T-t)-}$, where $T$ is some given time and $X_t$ is a stationary \pdmp. We obtain the parameters of the reversed process, like the jump intensity and the jump measure.}

\section{Introduction}\label{intro}

The aim of this paper is to introduce the time reversal of a general piecewise deterministic Markov process (PDMP). Given a stationary version $X_t$  of such a process we let $\ad X_t=X_{(T-t)-}$ and study the characteristics of this reversed process.

 The concept of reversing the direction of time for Markov processes can be found already in the works of Kolmogorov (\cite{kolmo1,kolmo2} and even earlier \cite{schroedinger}). The general idea is to study the process $\ad X_t=X_{(T-t)-}$, where $T>0$ is either a fixed point in time or a suitably defined random time. However, it is not clear whether the process $\ad X_t$ is a Markov process at all and if it is, whether properties like time-homogeneity, the strong-Markov property and others hold. Such questions have been answered until the 1970s  in several publications regarding time reversion (see e.g. \cite{Naga, Nag2,chungwalsh,walsh}). Other researchers applied time reversal to special classes of Markov processes, like L\'evy processes (\cite{JAPR}), stochastic networks (\cite{ Kellybook}), birth and death processes (\cite{tanaka89}) and Markov chains (\cite{norris}).

\pdmp s evolve deterministically on an open subset of $\sR^d$, interrupted by random jumps that happen either inside the state space or at the boundary. Piecewise deterministic paths can be observed for Markov processes in a variety of applications.  We  mention risk process
(\cite{HT, Rolski, DE,  dassjang, EMS}), growth collapse and stress
release models (\cite{BV, Ver, Zheng, Bea, AL-WOLFGANG}), queueing models (\cite{BrowneSigman, onoff}), earthquake models (\cite{ogata}), repairable systems (\cite{LastSzekli}), storage models (\cite{Cinlar1,HarRes,Harris2}) and TCP data transmission (\cite{AL-TRANS, AL-ISO, dumas}). The mathematical framework that is used in this paper was introduced by Davis \cite{Davis, Davis0}. Other approaches to PDMPs and related processes can be found in \cite{jask,wobst, schal1,schal2}, see also \cite{bono, Costa, DC, palmrol}. However, it seems that time-reversal has not been a subject of detailed study for PDMPs, at least not in a general framework (see \cite{fagg} for time-reversal arguments for a special subclass). To fill this gap we study the reversed process of a general PDMP, allowing also for what is called a non-empty active boundary, i.e. forced jumps that occur whenever the process hits the boundary of its state space.

This paper is divided into three sections. In the first section, after recalling the definition of \pdmp s, we investigate the imbedded Markov chains $\seqi{\BEF}{k}$ and $\seqi{\AFT}{k}$ obtained by observing the process just before and right after the jumps. Among other things we derive a formula for the stationary distribution of $\IMB_k=(\BEF_k,\AFT_k)$. Also in Section 1 we derive some sufficient conditions to ensure that the stationary distribution $\nu$ of the original PDMP is absolutely continuous.

In Section 2 we define the reversed Markov process $\seq{\ad X}{t}$ (we use an asterisk to denote variables related to the reversed process) and show that it is a PDMP. Moreover, we prove that $\sta(A,B)=\ad\sta(B,A)$, where $\sta$ and $\ad\sta$ denote the stationary  distribution of $\IMB_k$ and $\ad\IMB_k$. More specifically we find the crucial relation
\begin{align}
\ml x{dy}(\la(x)\nu(dx)+\sigma(dx))
=
\adml y{dx}(\ad\la(y)\nu(dy)+\ad\sigma(dy)),
\end{align}
where $\sigma$ denotes the measure that is equal to the average number of visits to the boundary.
This formula  can be used to derive the jump intensity $\adla(x)$ and the jump measure $\adml xA$ of the reversed process
-- one of the main aims of the paper. We also rediscover that in our setting the well known property of the generator $\ad\gen$ of $\ad X$ to be the adjoint of $\gen$ holds.

Section 3 is devoted to PDMPs on the real line, that is \pdmp s with state space $E\subseteq\sR$. We derive an integral equation for the stationary distribution and study some special cases.

We will introduce certain conditions for the process to hold. More specifically, we have conditions (A) to ensure that $X_t$ is a proper PDMP, conditions (B) and (D) to be able to revert the process and condition (C) to obtain an absolutely continuous stationary distribution.

\subsection{Preliminaries}\label{prel}
Throughout the article we use the following notations. We let $\leb_{d}$ denote be the $d$-dimensional Lebesgue measure and write $\mu_1\abscon\mu_2$ if a measure $\mu_1$ is absolutely continuous with respect to another measure $\mu_2$. In this case the symbol $\radons{\mu_1}{\mu_2}$ stands  for the Radon-Nikodym derivative. If $\mu_2$ is the Lebesgue-measure we shortly write $\mu_1'$ instead of $\radons{\mu_1}{\leb}$. The same notation $f'$ is used for the Radon-Nikodym derivative of an absolutely continuous function $f$. Given some measure $\mu$ we denote $\res{\mu}{A}(B)=\mu(A\cap B)$ the restriction of $\mu$ to $A$. 

We use the abbreviations $\EEE{\mu}$ and $\PPP{\mu}$ for the expectation and probability, given that the process $X_t$ starts with initial probability distribution $\mu$ (which will often be the stationary distribution $\nu$). In particular, if $\mu(\{x\})=1$ then we write $\EEE{x}$ and $\PPP x$. Given a set $A\subseteq\sR^d$, $\borel{A}$ indicates the Borel-$\sigma$-field of subsets of $A$.

The typical feature of a \pdmp\ is of course its eponymous path that consist of random jumps and piecewise deterministic segments. The jumps are steered by a jump intensity function, allowing the jump times to depend on the current state of the process, and a jump measure, determining the distribution of the destination of the random jumps. Additionally the process is allowed to change its state continuously between the jumps. In general, if $\hh(x,t)$ denotes the position of the process at time $t$, given that there were no jumps and that the process started in $x$, the appropriate condition on $\hh$ in order to obtain a deterministic time homogeneous Markov process is to form a flow, that is to fulfil the relation $\hh(\hh(x,t),s)=\hh(x,s+t)$ for $s,t\geq 0$ (see e.g. \cite{jacobsen,schal2}). If one wants to keep it general, this would be the only restriction to the deterministic paths and no further regularity and continuity conditions have to be imposed.

In this paper we instead follow the more restrictive (and more popular) set-up demonstrated in the book of Davis \cite{Davis} (see also \cite{Costa,Davis0, Rolski}), where the deterministic paths of the process are governed by a differential equation as follows. Let $\sEE_i\subseteq\sR^{d_i}$, $d_i\in\sN$, $i\in K=\{1,2,\ldots,k\}$, $k\in\sN$ be a collection of open sets and let  $\boundary_i$ denote the boundary of $\sEE_i$. We assume that on  each component $\sEE_i$ there is a vector field
 defined with unique integral curve $\hh_i:\sEE_i\times\sR\to \sEE_i$, having no explosions. The idea is that, given that $X_0=(x,i)$ with $x\in\sE_i$ and there was no jump during $[0,t]$, the position of $X_t$ is given by $\hh_i(x,t)$. We assume that there are functions $r_i=(r_{i,1},\ldots, r_{i,d_i}):\sEE_i\to\sR$, locally Lipschitz-continuous, such that for all $x\in\sEE_i$, the integral curve $\hh(x,t)$ is the unique solution of the ordinary differential equation
$\pd t\hh_i(x,t)=r_i(\hh_i(x,t))$. For differentiable functions $f:\sE_i\to\sR$ we can represent the integral curve using the Lie derivative
\begin{align*}
\xen _i f(x)\
=\lim_{t\to 0}\pd tf(\hh_i(x,t))
=\sum_{j=1}^{d_i} r_{i,j}(x)\frac{\partial }{\partial x_j}f(x).
\end{align*}
More generally, we understand the symbol $\xen_i f(x)$  as a solution of
\begin{align*}
f(\hh_i(x,t))=f(x)+\int_0^t \xen _i f(\hh(x,s))\,ds,
\end{align*}
allowing $f$ to be only absolutely continuous. 
The process  $\seq Xt$ evolves on a subset of the disjoint union of the $\sEE_i\cup\boundary_i$ and the elements of $K$ represent the outer states of the process.  In order to give a complete specification of the  actual state space of the process we distinguish between the following subsets of the boundary $\boundary_i$. First, the so called active boundary of points that can be reached from within $\sEE_i$,
\begin{eqnarray*}
\act_i&=&\{z\in\boundary_i:z=\hh(x,t)\quad\mbox{for some $x\in\sEE_i$, $t>0$}\}
\end{eqnarray*}
and secondly, what we call (in lack of a more appropriate term) the passive boundary, the set of points on the boundary from where points in $\sEE_i$ can be reached:
\begin{eqnarray*}
\acm_i&=&\{z\in\boundary_i:z=\hh(x,-t)\quad\mbox{for some $x\in\sEE_i$, $t>0$}\},
\end{eqnarray*}
of course, as the asterisk indicates, this set will serve as active boundary for the reversed process. We also define
$\sE_i=\sEE_i\cup\acm_i$ and $\sEA_i=\sEE_i\cup\act_i$. For each set  $S_i\in\{\act_i,\acm_i,\sE_i,\sEA_i\}$ we define the disjoint union  $S=\{(x,i):i\in K,x\in S_i\}$. For example the state space of $X_t$ becomes
\begin{eqnarray*}
\sE=\{(x,i):i\in K,x\in\sE_i\},
\end{eqnarray*}
while the active boundary is $\act=\{(x,i):i\in K,x\in\act_i\}$. Before we proceed we agree on the following convention. We will throughout the paper omit the outer states, e.g. write $x\in\sE$ instead of $(x,i)\in\sE$, or $\xen  f(x)$ instead of $\xen _i f(x)$. Only if the notations are necessary to avoid ambiguity we well indicate the outer states.

As explained in the introduction, the process $X_t$ jumps at certain random times $\seqi{T}{i}$. This will either happen from within $\sE$ (voluntary jumps) or if the process hits the active boundary $\act$ (forced jumps). The times at which forced jumps occur are denoted by $\seqi{\fac T}{i}$. Let $\tact(x)=\inf\{t\geq 0: \hh(x,t)\in\act\}$ denote the first time the latter happens, given the process starts in $x\in\sE$ and no jumps occurred (with the usual convention that the infimum is equal to $\infty$ if there is no such $t$). Two random mechanisms determine the jumps:
the jump intensity is a non-negative, continuous measurable function $\la:\sE\to[0,\infty)$, with the interpretation that the probability of a jump during $[t,t+h]$, given that the process is in the state $x$, is $\la(x)h+o(1)$ and the probability of more than one jump is $o(1)$ as $h\to 0$.
Formally this will follow from the continuity of $\la$ and from the representation of the probability distribution of  the first jump time $T_1$  given by
\begin{equation}\label{Tone}
\PX x{T_1>t}=\ind{t<\tact(x)}\La{t}{x},
\end{equation}
where we use the abbreviation  $\La{t}{y}=\exp(-\int_0^t \la(\hh(y,u))\,du)$.
Secondly, the jump measure $\ml xA$ determines the probability of a jump from $x\in\sEA$ into a measurable set $A\subseteq\borel{\sE}$.
Let $N_t=\sup\{n:T_n\leq t\}$ denote the number of jumps (forced and voluntary) and $\fac N_t=\sup\{n:
\fac T_n\leq t\}$ the number of forced jumps occurring during $[0,t]$.

Throughout we assume the following standard conditions (c.f. \cite{Davis}):

\begin{Condition}\ \label{AA}
\begin{enumerate}
\itc{AA}1\label{AA1} For all $x\in\sE$ there is an $\veps(x)>0$, such that $\int_0^{\veps(x)} \la(\hh(x,s))\,ds<\infty$ (local integrability of $\la$ along $\hh$);
\itc{AA}2\label{AA2} $\int_0^{\infty} \la(\hh(x,s))\,ds=\infty$ for all $x\in\sE$ with $\tact(x)=\infty$ (first jump happens in finite time);
\itc{AA}3\label{AA3} $\EX x{N_t}<\infty$ for all $x\in\sE$ and all $t>0$ (there are no explosions);
\itc{AA}{4}\label{AA4} $\ml x{\{x\}}=0$ for all $x\in\sE$ (no jumps of size zero);
\end{enumerate}
\end{Condition}

Additionally we assume further conditions to ensure that we can define a proper reversed process and to be able to derive formulas for the parameters of the reversed process. In what follows we define
\begin{align*}
\dis hB=\{x\in\sE:\tim xB<h\}
\end{align*}
for any set
 $B\in\borel{\sE\cup\Gamma}$ to denote the locations from which the process can reach $B$ within $h$ time units.

\begin{Condition} In this paper we always assume the following conditions. \label{BB}
\begin{enumerate}
\itc{BB}1\label{BB1} There exists a stationary distribution $\nu$ of $X_t$ on $\sE$;
\itc{BB}2\label{BB2} $\int_{\sE} \la(x)\nu(dx)<\infty$;
\itc{BB}3\label{BB3} $\la$ is continuous;
\itc{BB}4\label{BB4} For any $A\in\borel\sE$ and every $x\in\sE$ we have $\mu_{\hh(x,h)}(A)\to \mu_x(A)$ as $h\to 0$.
\itc{BB}5\label{BB5} $\mu_{x}(\dis h{\act})\to 0$ uniformly for  $x\in \act$ as $h\to 0$.
\end{enumerate}
\end{Condition}

\begin{remark}
Condition \CD{BB}{BB1} assumes the existence of a stationary distribution of the process. This will be crucial for the definition of the time reversal in Section \ref{reverse}. Giving appropriate conditions for a general  PDMP to possess a stationary distribution, to be ergodic, positive recurrent or Harris recurrent is a difficult if not impossible task. Different efforts have been made to solve these problems for special cases (\cite{lastergo, Zheng, Costa}). We take the liberty to keep away from these intricate questions and just assume our process to posses a unique stationary distribution, being aware of the difficulties that we leave untouched.
\end{remark}

\begin{remark}
Condition \CD{BB}{BB5} prevents the process from becoming cascading near the boundary, i.e. to jump more and more often from $\Gamma$ into a smaller and smaller neighbourhood of $\Gamma$.
\end{remark}

\begin{remark}
With these conditions holding there are no hybrid jumps, that is, we do not allow the situation where there is an $x\in\sE$ such that $X_t$ jumps with a certain probability $p(x)$ once it reaches $x$ and with probability $1-p(x)$ it stays on the flow. This restriction becomes important for the reversed process.
\end{remark}

We will also need the following substitution rules. Let $\tim xy=t$ if $y=\hh(x,t)$ for some $t\geq 0$ and $\tim xy=\infty$ else. Then $\tim xy$ represents the time the process needs to run from $x$ to $y$ if no jumps occur. Let $\cv xy=\cup_{t\geq 0}\{\hh(x,s)\in\sE:s\in [0,t],\; {\rm and}\; y=\hh(x,t)\}$ and $\cvv x=\{\hh(x,s)\in\sE:s\in [0,\tact(x))\}$ denote the curve segments starting from $x$ and ending at $y$ and at the active boundary, respectively. Then we can rewrite the line integral over $\cv xy$ as
\begin{align}\label{u1}
\int_{\cv xy} f(u)\,du=\int_0^{\tim xy} f(\hh(x,s))|r(\hh(x,s))|\,ds,
\end{align}
where $|r(u)|=(\sum_{i=1}^d r_i^2(u))^{1/2}$. Similarly, provided that $|r(u)|>0$ on $\cv xy$, the integral over $f(\hh(x,s))$ with $s\in[0,\tim xy]$ can be written as
\begin{align}\label{u2}
\int_0^{\tim{x}y} f(\hh(x,s))\,ds=\int_{\cv xy} \frac{f(u)}{|r(u)|}\,du.
\end{align}
It follows in particular that 
$\tim{x}y=\int_{\cv xy} 1/{|r(u)|}\,du$.

Let $\mes$ denote the class of measurable functions $f:\sEA\to\sR$ and $\mabs$ the class of those members $f\in\mes$, for which  $t\to f(\hh(x,t))$ is absolutely continuous (so $f$ is absolutely continuous along the flow). We define the linear operator
\begin{align}
\gen f(x)&=\xen  f(x)+\la(x)\qen f(x)\label{generator},
\end{align}
where $\qen f(x)=\int_{\sE} \kla{f(y)-f(x)}\,\ml x{dy}$. $\gen$ is the
full generator of the Markov process $X_t$ and
we denote by $\dom\gen$ the domain  of $\gen$, so that for $f\in \dom\gen$ the process
$f(X_t)-f(X_0)-\int_0^t \gen f(X_s)\,ds$ becomes a martingale.
These functions are characterized in \cite[Th. (26.14) and Rem. (26.16)]{Davis}.
Note  that the class of bounded functions $f\in\mabs$ with $\qen f(x)=0$ for all $x\in\act$ is contained in the domain. We let $\bacon$ denote the bounded functions in $\mabs$. We note that for all $f\in\bacon$
\begin{align}
&f(X_t)-f(X_0)-\int_0^t \gen f(X_s)\,ds-\int_0^t \qen f(X_{s-})\,d\fac N_s\label{marti}
\end{align}
is also a martingale.

\subsection{Stationarity}

The process $X_{t-}$ will hit the active boundary $\act$ at certain times $\seq{\fac T}n$ and is then forced to jump. Let $\fac N_t(B)$ denote the number of hits of a set $B\in\borel{\act}$ during the time interval $[0,t]$ and $\fac N_t=N_t(\act)$. It is shown in \cite[Theorem 34.15]{Davis} (\cite[Proposition 2]{Costa}) that there exists a finite measure $\sigma$ on the active boundary $\act$ such that $\sigma$ measures the time-average of the number of visits to $B$, namely
$\sigma(B)=\EX{\nu}{\fac N_t(B)}/t$ for every $t>0$. That is, for any bounded function $f\in\mes$ and any $t>0$,
\begin{eqnarray}\label{boundary}
\int_{\act}f(x)\sigma(dx)=\frac{1}{t}\EX\nu{\int_0^tf(X_{s-})\,d\fac N_s}.
\end{eqnarray}
Many of the upcoming results of this paper are based on the following crucial relation. Taking expectations in \eqref{marti} with respect to $\nu$ and taking into account that $\EX\nu{f(X_t)}=\EX\nu{f(X_0)}$,  we obtain 
$t\EX\nu{\gen f(X_s)}+\EX\nu{\int_0^t \qen f(X_{s-})\,d\fac N_s}=0$ for every  $f\in\bacon$,
so that we have the identity (see \cite[Theorem 34.19]{Davis}):
\begin{equation}\label{zero}
\int_{\sE} \gen f(x)\nu(dx)+\int_{\act}\qen f(x)\sigma(dx)=0.
\end{equation}
Equation \eqref{zero} is the usual starting point to find expressions of $\nu$ in terms of the parameters $\la$ and $\mls x$ of the process. In general however there is not much hope to find explicit formulas and only for a few PDMP models the stationary distribution $\nu$ is explicitly known. Even in the one-dimensional case often all that can be done is to derive integro-differential equations for $\nu$ (see Section \ref{onedi} for examples) and even these equations give rise to challenging problems themselves.

Another interesting question is, whether it is possible to further describe the measure $\sigma$ and express it in terms of the stationary measure $\nu$ on $\sEE$. In fact,  $\sigma$ is determined by the values of $\nu$ on  arbitrary small neighbourhoods of $\act$ and for $B\in\borel{\act}$ the measure $\sigma(B)$ is given by the limit probability to find $X_t$ in short distance to $B$  (measured in time units needed to reach $B$). Let $\tim xB=\tim xy$ if $y=\hh(x,t)\in B$ for some $t\geq 0$ and $\tim xB=\infty$ else.

\begin{Proposition}\label{sigm} For $B\in\borel{\act}$ the measure $\sigma(B)$ is given by the limit
\begin{align}
\sigma(B)=\lim_{h\to 0}\frac{\nu(\dis hB)}{h}.\label{rice}
\end{align}
\end{Proposition}
\begin{proof}
The function $f_h:\sEA\to[0,1]$
\begin{align*}
f_h(x)=\bkla{1-\frac{\tim x B}{h}}\cdot\ind{x\in\dis hB}
\end{align*}
 is absolutely continuous, bounded and vanishes on $\sE\setminus \dis hB$;  in particular $f_h\in\bacon$ and
\begin{align*}
\xen  f_h(x)=\lim_{t\to 0}\frac{d}{dt}f_h(\hh(x,t))=-\frac{1}{h}\ind{x\in \dis hB}\lim_{t\to 0}\frac{d}{dt}\tim {\hh(x,t)} B=\frac{\ind{x\in \dis hB}}{h}.
\end{align*}
Application of \eqref{zero} to $f_h$ then yields
\begin{align*}
&\frac{1}{h}\nu(\dis hB)+\int_{\dis hB}\la(x)\bkla{\int_{\dis hB}\ml x{dy}-f_h(x)}\nu(dx)
\\&\quad\quad=-\int_{\act}\bkla{\int_{\dis hB}f_h(y)\ml x{dy}-f_h(x)}\sigma(dx),
\end{align*}
from which it follows that
\begin{align*}
\frac{1}{h}\nu(\dis hB)
=\sigma(B)-\int_{\act}\int_{\dis hB}f_h(y)\ml x{dy}\sigma(dx)-\int_{\dis hB}\la(x)\bkla{\ml x{\dis hB}-f_h(x)}\nu(dx).
\end{align*}
Note that
\begin{align*}
&\left|\int_{\act}\int_{\dis hB}f_h(y)\ml x{dy}\sigma(dx)-\int_{\dis hB}\la(x)\bkla{\ml x{\dis hB}-f_h(x)}\nu(dx)\right|
\\&\quad\leq \int_{\act}\ml x{\dis hB}\sigma(dx)+\int_{\dis hB}\la(x)\nu(dx).
\end{align*}
As $h\to 0$ the integrals on the right tend to zero by dominated convergence, where we used condition \CD{BB}{BB2} and $\nu(\Gamma)=0$.
\end{proof}

\subsection{The embedded processes}

To gain insight into the behaviour of the process it is useful to study the continuous time process $\seq Xt$ sampled at the jump times $\seqi Ti$. To this end we define the two-dimensional process $\IMB_k=(\BEF_k,\AFT_k)$, where  $\BEF_k=X_{T_{k-}}$ and $\AFT_k=X_{T_k}$ denote the embedded discrete time Markov processes, obtained by observing $X_t$ just after and right before the jumps. Note that due to the forced jumps the state space of $\BEF$ is $\sEA$, whereas $\AFT$ lives on $\sE$, so that $\sEI=\sEA\times\sE$ is the state space of $\IMB$.

If $X_t$ is stationary then still $\BEF_k$ and $\AFT_k$ are usually not stationary and vice versa. However,  if $\BEF_k$ is stationary then so is $\AFT_k$. If $\nu$ is the stationary distribution of $X_t$, then we will show how to find a stationary distribution for $\IMB_k$.

Let for $A\in\borel\sEA$ and $B\in\borel{\sE}$
\begin{eqnarray*}
\por hAB&=&\PX \nu{X_0\in A,\; X_{h}\in B,\;  T_1\in[0,h]}
\end{eqnarray*}
denote the probability that (under $\nu$) $X_0\in A$, $X_{h}\in B$ and the process is going to jump
during the time interval $[0, h]$.

\begin{Proposition}
For all $A\in\borel\sEA$ and $B\in\borel{\sE}$ the limit ${\por hAB}/{h}$ exists as $h\to 0$  and it is given by:
\begin{align}\label{sh}
\xix(A,B)=\lim_{h\to 0}\frac{\por hAB}{h}=\int_{A\cap\sE}\la(x)\ml x{B}\nu(dx)+\int_{A\cap\act}\ml x{B}\sigma(dx).
\end{align}
\end{Proposition}

\begin{proof} Suppose first that $B$ is open. The probability $\por hAB$ is a sum of the probabilities of the following four events:

{\noindent 1)\;} We have $X_0\in (A\cap\sE)\setminus\dis h(A\cap\act)$, i.e. $X_0$ is not close to the boundary and a jump occurs within $h$ time units, say at time $s$. Moreover, $X_s\in \dis{h-s}B$ and no further jumps occur, so that the process will end up in $B$ at time $h$, as desired. The probability  of this event is given by
\begin{align*}
p_1=\int\limits_{(A\cap\sE)\setminus\dis h(A\cap\act)}\hspace{-3ex}
\int_0^h
\int_{\dis {h-s}{B}}\La{h-s} z\ml{\hh(x,s)}{dz}\La{ds}{x}\,\nu(dx).
\end{align*}
Since $B$ is open, we have that $\ind{z_h\in\dis {h-s}{B} }\La{h-s}{z_h}\to \ind{z\in B}$ for any sequence $\{z_h\}\in\sE$ with $z_h\to z$, $z\in\sE$ as $h\to 0$ and by condition \CD{BB}{BB4} it follows that $\ml{\hh(x,s)}{\dis {h-s}B}\to \ml{x}{B}$ as $h\to 0$. Hence, by Theorem 5.5 in \cite{bill} $\int_{\dis {h-s}{B}}\La{h-s} z\ml{\hh(x,s)}{dz}\to \ml xB$ as $h\to 0$. Since $\La{ds}{x}=\la(\hh(x,s))\La{s}{x}\,ds$ and $\La{s}{x}\to 1$ as $h\to 0$, it follows from the continuity of $\la$ (condition \CD{BB}{BB3}) that
\begin{align*}
p_1 &= h\hspace{-4ex}\int\limits_{(A\cap\sE)\setminus\dis h(A\cap\act)}\hspace{-4ex}\la(x)\ml{x}{B}
\,\nu(dx)+o(h)
= h\int\limits_{A\cap\sE}\la(x)\ml{x}{B}
\,\nu(dx)+o(h).
\end{align*}
The last equality is a consequence of the fact that $\nu(\dis h(A\cap\act))\to 0$ as $h\to 0$ since $\nu(\Gamma)=0$.

{\noindent 2)\;} The event that $X_0\in \dis h(A\cap\act)$, a unforced jump occurs at time $s$ before the process reaches $\Gamma$,  $X_s\in \dis{h-s}B$ and no further jumps occur is given by
\begin{align*}
p_2=\int_{\dis h(A\cap\act)}\int_0^{\tact(x)}\int_{\dis {h-s}{B}}\La{h-s} z\ml{\hh(x,s)}{dz}
\La{ds}{x}\,\nu(dx).
\end{align*}
Since $\tau(x)\to 0$ and  $\nu(\dis h(A\cap\act))\to 0$ as $h\to 0$, it follows by similar arguments as before that $p_2=o(h)$.

{\noindent 3)\;} $X_0\in \dis h(A\cap\act)$, the process reaches $\Gamma$ at time $\tact(x)$,  $X_s\in \dis{h-s}B$ and no further jumps occur. The probability is given by
\begin{align*}
p_3=\int_{\dis h(A\cap\act)}\int_{\dis {h-\tact(x)}{B}}\La {h-\tact(x)}z\ml{\hh(x,\tact(x))}{dz}
\La{\tact(x)}{x}\,\nu(dx).
\end{align*}
As in 1) we have that $\ind{z\in\dis {h-\tact(x)}{B}}\La {h-\tact(x)}z$ tends to $\ind{z\in B}$ for any sequence $\{z_h\}\in\sE$ with $z_h\to z$, $z\in\sE$ as $h\to 0$ and $\ml{\hh(x,\tact(x))}{\dis {h-\tact(x)}B}\to \ml{x}{B}$. Consequently
\begin{align*}
\int_{\dis {h-\tact(x)}{B}}\La {h-\tact(x)}z\ml{\hh(x,\tact(x))}{dz}
\La{\tact(x)}{x}\to \ml{x}{B}.
\end{align*}
From Proposition \ref{sigm} it follows that $\nu(\dis h(A\cap\act))=h\sigma(A\cap\act)+o(h)$. Hence
\begin{align*}
p_3=h\int_{A\cap\act} \ml{x}{B}\sigma(dx)+o(h).
\end{align*}

{\noindent 4)\;} We finally investigate the case, where more than one jump occurs during the time period $[0,h]$. The probability of this event is $o(h)$ as $h\to 0$. This follows immediately from the definition, if the process does not reach the boundary. If on the other hand $X_0\in\dis h\Gamma$ (the probability being $O(h)$) then the probability to jump to $\dis h\Gamma$ again is $o(1)$ by Condition \CD{BB}{BB5}.

Altogether we obtain
\begin{align}
\por hAB=h\int_{A\cap\sE}\la(x)\ml{x}{B}\,\nu(dx)+h\int_{A\cap\act} \ml{x}{B}\sigma(dx)+o(h).\label{por}
\end{align}
This completes the proof for the case, where $B$ is open. The general case follows using classical arguments.
\end{proof}

The total mass $\nor\xix=\xix(\sEA,\sE)=\int_{\sE} \la(u)\,\nu(du)+\sigma(\act)$ is finite due to \CD{BB}{BB2} and it
 has been shown in \cite[Proposition 5]{Costa} (and \cite[Theorem (34.21)]{Davis}) that $\nor\xix$
is strictly positive. Consequently, we can define a probability measure
\begin{align}\label{sta}
\sta(A,B)=\afrac{\xix(A,B)}{\nor\xix}
=\frac{1}{\nor\xix}\Bkla{\int_{A\cap\sE}\la(x)\ml x{B}\nu(dx)+\int_{A\cap\act}\ml x{B}\sigma(dx)}
\end{align}
on $\sEI$.

\begin{Theorem}\label{ttt}
$\sta$ defined in \eqref{sta} is a stationary measure for the Markov chain $\IMB$.
\end{Theorem}

\begin{remark}
Compare with  \cite[Theorem (34.21)]{Davis} and \cite[Theorem 1]{Costa}, where it is shown that the marginal distribution $\pi(\sE,\cdot)$ is a stationary distribution for $\AFT$. The present proof follows partly along the lines of the proofs presented there.
\end{remark}

\begin{proof}[Proof of Theorem \ref{ttt}]  Let $\ten f(y)=\EE{f(\BEF_k,\AFT_k)|\AFT_{k-1}=y}\in\bacon$ for some function
$f\in \bacon$.
Conditioning with respect to the first jump time of the process we obtain:
\begin{eqnarray*}
\ten f(y)
&=&\int_0^{\tact(y)}\int_{\sE}f(\hh(y,s),u)\,\ml{\hh(y,s)}{du} \la(\hh(y,s))\La{s}{y}\,ds
\nline+\int_{\sE}f(\hh(y,\tact(y)),u)\,\ml{\hh(y,\tact(y))}{du} \La{\tact(y)}{y}.
\end{eqnarray*}
It follows from the flow property $\hh(\hh(x,t),s)=\hh(x,s+t)$, that
\begin{eqnarray*}
\ten f(\hh(y,t))
&=&\int_t^{\tact(y)}\int_{\sE}f(\hh(y,s),u)\,\ml{\hh(y,s)}{du} \la(\hh(y,s))\frac{\La{s}{y}}{\La{t}{y}}\,ds
\nline+\int_{\sE}f(\hh(y,\tact(y)),u)\,\ml{\hh(y,\tact(y))}{du} \frac{\La{\tact(y)}{y}}{\La{t}{y}}
\\&=&\frac{1}{\La{t}{y}}\Bkla{\ten f(y)-\int_0^{t}\int_{\sE}f(\hh(y,s),u)\,\ml{\hh(y,s)}{du} \la(\hh(y,s))\La{s}{y}\,ds}.
\end{eqnarray*}
Since $\xen  \ten f(y)=\lim_{t\to 0}\frac{d}{dt}\ten f(\hh(y,t))$, we obtain
$
\xen  \ten f(y)=\la(y)\kla{\ten f(y)-\int_{\sE}f(y,u)\,\ml{y}{du}}.
$
It follows that the generator applies to $\ten f$ in the following way,
\begin{eqnarray}
\gen \ten f(y)&=&\la(y)\Bkla{\ten f(y)-\int_{\sE}f(y,u)\,\ml{y}{du}+\qen \ten f(y)}
\nonumber\\&=&\la(y)\Bkla{\ten f(y)-\int_{\sE}f(y,u)\,\ml{y}{du}+\int_{\sE} \kla{\ten f(u)-\ten f(y)}\,\ml y{du}}
\nonumber\\&=&\la(y)\int_{\sE}(\ten f(u)-f(y,u))\,\ml{y}{du}.\label{uu}
\end{eqnarray}
Next we show that $\xix$ is an invariant measure for  $\IMB$, i.e. $\xix{\ten f}=\xix f$, where we use the usual abbreviation
$\xix f=\int_{\sEI} f(x,y)\xix(dx,dy)$ and $\xix{\ten f}=\int_{\sEI} \ten f(y)\xix(dx,dy)$. It follows from \eqref{zero}  and \eqref{uu} that
\begin{eqnarray*}
\int_{\sE} \la(y)\int_{\sE}(\ten f(u)-f(y,u))\,\ml{y}{du}\,\nu(dy)+\int_{\act} \int_{\sE}(\ten f(u)-\ten f(y))\,\ml y{du}\,\sigma(dy)=0.
\end{eqnarray*}
Hence, using the identity $\ten f(y)=\int_{\sE}f(y,u)\,\ml{y}{du} $ for $y\in\act$, we have:
\begin{eqnarray*}
&&\int_{\sE} \la(y)\int_{\sE}\ten f(u)\,\ml{y}{du}\,\nu(dy)+\int_{\act} \int_{\sE}\ten f(u)\,\ml y{du}\,\sigma(dy)
\nline=\int_{\sE} \la(y)\int_{\sE} f(y,u)\,\ml{y}{du}\,\nu(dy)+\int_{\act}\int_{\sE}f(y,u)\,\ml{y}{du}\,\sigma(dy)
\end{eqnarray*}
proving that $\xix{\ten f}=\xix f$.
\end{proof}

It follows immediately that a stationary distributions of $\BEF_k$ (the state of the process just before the jump) is given by
\begin{eqnarray}
\stbef(A)&=&\frac{1}{\nor{\xix}}\Bkla{\int_{A\cap\sE}\la(x)\nu(dx)+\sigma(A\cap\act)},\label{piw}
\end{eqnarray}
for $A\in\borel{\sEA}$. The restricted measure $\res{\stbef}{\sE}$ is absolutely continuous with respect to $\nu$ with Radon-Nikodym derivative $\la$, whereas on the active boundary $\res{\stbef}{\act}$ is a constant  multiple of the measure $\sigma$. Note that for an empty active boundary and a constant jump rate the above relation shows that the PASTA  property  (Poisson Arrivals See Time Average) $\stbef(A)=\nu(A)$ holds.

By Theorem \ref{ttt} a stationary distribution for the observations $\AFT_k$ of the process right  after the jumps is given by
\begin{eqnarray}
\staft(B)&=&\int_{\sEA} \ml x{B}\stbef(dx)\label{piq}
\\&=&\frac{1}{\nor{\xix}}\Bkla{\int_{\sE}\la(x)\ml x{B}\nu(dx)+\int_{\act}\ml x{B}\sigma(dx)},\nonumber
\end{eqnarray}
for $B\in\borel\sE$, confirming the results in \cite{Davis, Costa}. Note that  the formula for $\staft$ follows from \eqref{piw} after conditioning on the jump size.

\subsection{Absolute continuity of the stationary measure}

It is interesting to ask under which conditions $\nu$ is an absolutely continuous measure (w.r.t. Lebesgue measure).
This property jointly with absolutely continuity of the jump measure or stationary measure $\pi_Q$
simplifies the  identification of the parameters of the reversed process.
For one-dimensional models on the real line $\nu$ was found to be absolutely continuous under very mild conditions, see \cite{Zheng, BV, borovlast, lastergo,lastrice, Costa}. This is in accordance with countless observations of absolutely continuous stationary distributions for PDMP type stochastic models in the literature (e.g. \cite{Asmussen, Bea, Brockwell, HarRes, rene, barpar, AL-ISO}). However, in general, i.e. for higher dimensions, it seems very hard to give necessary and sufficient criteria.

We introduce one more condition on the process $X_t$.

\begin{Condition}\label{COC} $\EX{\staft}{T_1}<\infty$.
\end{Condition}

\begin{remark} 
The condition is satisfied in the vast majority of cases. The reason to introduce it is the following useful converse to \eqref{piq}. If  $\EX{\staft}{T_1}<\infty$ then the stationary measure $\nu$ can be reconstructed from $\staft$ by an argument from regeneration theory. Note that under $\staft$ the time until the first jumps occurs forms a cycle, hence, by \cite[Corollary VII 1.4]{Asmussen},
\begin{align}
\EX{\nu}{f(X_0)}=\frac{\EX{\staft}{\int_0^{T_1}f(X_s)\,ds}}{\EX{\staft}{T_1}};\label{regener}
\end{align}
(c.f. with  \cite[equation (12)]{Costa}) for bounded continuous functions $f$.
\end{remark}

Recall that the state space $\sE$ consists of a disjoint union of components $\sE_i$: $\sE=\{(x,i):i\in K,x\in\sE_i\}$.
 The following proposition gives a criterion for absolute continuity of $\nu$ on one of these $\sE_i$, $i\in K$.

\begin{Proposition}\label{prop} Suppose that, additionally to \cd{COC}, $r(x)\not=0$ for all $x\in\sE_i$. If $d(i)=1$ then $\nu$ is absolutely continuous on $\sE_i$. In general, $\nu$ is absolutely continuous on $\sE_i$ if $\staft$ is absolutely continuous there and $r$ is continuously differentiable.
\end{Proposition}

\begin{proof}
It follows from \eqref{regener} that
\begin{align}
\nu(A)\leq c\EX{\staft}{\int_0^{T_1}\ind{X_s\in A}\,ds}\label{from}
\end{align}
with some constant $c>0$. If $d(i)=1$ and $r(x)\not=0$, then $\leb_1(A\cap \cvv x)=0$ for all Lebesgue null sets $A\in\borel{\sR}$ and in particular $\EX{x}{\int_0^{T_1}\ind{X_s\in A}\,ds}=0$. It follows from \eqref{from} that $\nu(A)=0$, showing that $\nu$ is absolutely continuous. In the general case, i.e. if $d(i)\geq 1$, suppose that $\staft$ is absolutely continuous. We have
\begin{align}
\EX{\staft}{\int_0^{T_1}\ind{X_s\in A}\,ds}
\leq\int_{C_A}\EX x{T_1} \;\staft(dx),\label{894}
\end{align}
where $C_A=\{x\in\sE_i:\int_0^{\tact(x)} \ind{\hh(x,s)\in A}\,ds\not=0\}$.  If we can show that $C_A$  is a Lebesgue null set, then it follows that $\staft(C_A)=0$, since $\staft$ is absolutely continuous. But then the right hand side of \eqref{894} is zero and absolutely continuity of $\nu$ follows from \eqref{regener}. To show that $\leb_{d(i)}(C_A)=0$, suppose the converse. Then
\begin{align}
0<\int_{\sE}\int_0^{\infty} \ind{s<\tact(x),\hh(x,s)\in A}\,ds\,dx
=\int_0^{\infty} \int_{\sE_i}\ind{s<\tact(x),\hh(x,s)\in A}\,dx\,ds.\label{mox}
\end{align}
If $r$ is continuously differentiable then $\zeta_t:x\mapsto \hh(x,-t)$ is also continuously differentiable (e.g. \cite[p.95]{hartman}). It follows that $\zeta_t$  maps Lebesgue null sets to Lebesgue null sets (see \cite[Proposition 26.3]{Yeh}) and that the integral of $\int_{\sE_i}\ind{s<\tact(x),\hh(x,s)\in A}\,dx$ over the indicator function of the image $\zeta_s(A)$ of $A$, must be zero, contradicting \eqref{mox}. Hence $\leb_{d(i)}(C_A)=0$ and  $\nu$ is absolutely continuous.
\end{proof}

\begin{remark}
The condition on $r$ can be relaxed. In fact it is enough to require that $\zeta_t$ fulfils the weaker condition
\begin{align}
\sup_{x\in\sE} \lim_{x\to y}\frac{|\zeta_t(x)-\zeta_t(y)|}{|x-y|}<\infty.
\label{moose}
\end{align}
In that case $\zeta_t$ fulfils what is called Lusin's condition (N), that is $\zeta_t$ maps null sets to null sets (\cite[Proposition 26.2]{Yeh}).
\end{remark}

\section{The reversed process}\label{reverse}
\subsection{Definition}
We now assume $X_t$ to be stationary, that is, the process starts with  initial distribution $\nu$ and then has the same distribution for all $t\geq 0$. We pick a fixed time $T>0$ and define the reversed process for $t\in[0,T]$  by $\ad X_t=X_{(T-t)-}$ (we indicate variables belonging to the reversed process with a star). Then $\ad X_t$ is a right-continuous stationary stochastic process with state space  $\sEA$ and initial distribution $\nu$. Obviously the active boundary of the reversed process is given by $\acm$ and conversely the passive boundary is now $\act$. It is known (see \cite{Naga} and \cite[Theorem 2.1.1]{Nag2}) that $\ad X_t$, constructed in this way, is again a time homogeneous Markov process. Obviously  $\ad X_t$ inherits from $X_t$ the property to have piecewise deterministic paths with random jumps and no explosions. So $\ad X_t$ would in fact be a PDMP if we could  show that $\ad X_t$ has a regular jump intensity $\adla$, in the sense that the conditions \CD{AA}{AA1} and \CD{AA}{AA2} for $\la$ have to be fulfilled also for $\adla$ (this is not self-evident, see the example below). Therefore, we have to impose two more conditions on the process $X_t$, namely $\staft$ has to be absolutely continuous on $\sEE$ (allowing a mass on $\acm=\sE\setminus\sEE$) with a locally integrable Radon-Nikodym derivative, which in Proposition \ref{wrc} below will  be seen to be our jump intensity function for the reversed process. 

\begin{Condition}\label{DD} From now on we assume:
\begin{enumerate}
\itc{DD}1\label{DD1}
$\res{\staft}{\sEE}\abscon\nu$ with Radon-Nikodym derivative $\beta(x)=\radons{\res{\staft}{\sEE}}{\nu}(x)$, $x\in\sEE$.
\itc{DD}2\label{DD2} For all $x\in\sE$ there is an $\veps(x)>0$, such that $\int_0^{\veps(x)} \beta(\hh(x,-s))\,ds<\infty$.
\end{enumerate}
\end{Condition}

\begin{Proposition}\label{wrc}
Under \cd{DD} the reversed process is a PDMP, fulfilling the conditions \CD{AA}{AA1}-\CD{AA}{AA4}, with intensity function given by
$\adla(x)=\nor{\xix}\beta(x)$.
\end{Proposition}

\begin{proof} According to \eqref{por} the probability $\adpor hB{\sEA}$ that $\ad X_t\in B\subseteq\sEE$ and that the process has a voluntary jump during $[t,t+h]$ is given by
\begin{align*}
\adpor hB{\sEA}&=h\cdot \int_{B}\adla(x)\nu(dx)+o(h).
\end{align*}
On the other hand, by construction of $\ad X_t$, $\adpor hB{\sEA}$ is equal to the probability $\por h{\sEA}B$ and hence equal to
\begin{align}
\por h{\sEA}B&=h\cdot \Bkla{\int_{\sE}\la(x)\ml xB\nu(dx)+\int_{\act} \ml xB \sigma(dx)+o(1)}\label{o1}
\\&=h\nor{\xix}\staft(B)+o(h)=h\int_B \nor{\xix}\beta(x)\,\nu(dx)+o(h),\nonumber
\end{align}
yielding $\adla(x)=\nor{\xix}\beta(x)$. \cd{DD} ensures that $\adla$ fulfills \CD{AA}{AA1}.  Moreover the condition \CD{AA}{AA2}, i.e. $\PP{\ad T_1<\infty}=1$, certainly holds. Indeed $\adla(x)$ is a proper intensity function for the PDMP $\ad X_t$.
\end{proof}

\begin{remark}
Since
\begin{eqnarray*}
\int_{\sEA}\adla(x)\nu(dx)=\nor{\xix}\int_{\sEA}\beta(x)\nu(dx)=\nor{\xix}\staft(\sEA)<\infty,
\end{eqnarray*}
it follows that \CD{BB}{BB2} always holds.
\end{remark}

\begin{remark}
\cd{DD} is not always fulfilled, as the following simple counterexample shows. Consider $X_t$ on $[0,1]$ with $\xen  f(x)=f'(x)$ (linear growth with rate one), $\act=\{1\}$, $\la(x)=0$ for $x\in[0,1)$, $\mu_{1}(A)=\frac{1}{2}(\ind{1/2\in A}+\ind{0\in A})$, so that jumps can go from $\act$ to $0$ and $1/2$ only, each with probability $1/2$. It  is not difficult to show that
$$\nu([0,x])=\frac{\ind{x<1/2}2x+\ind{x\geq 1/2}(4x-1)}{3},$$ hence the stationary measure is absolutely continuous, but $\staft$ is certainly not. What happens is that the reversed process $\ad X$ is not a proper PDMP since with a positive probability the process is forced to jump  when reaching $1/2$.
\end{remark}

\subsection{The parameters of the reversed process}

The following theorem shows that the stationary distribution of the reversed Markov chain $\ad\IMB_k=(\ad\BEF_k,\ad\AFT_k)$ is given by the measure $\sta$, but with reversed arguments.

\begin{Theorem}\label{main} For all $A\in\borel{\sEA},B\in\borel{\sE}$ we have $\sta(A,B)=\ad\sta(B,A)$.
\end{Theorem}

\begin{proof} The probability $\adpor hBA$ that $\ad X\in B$ and that the process jumps during $[t,t+h]$ into the set $A$ is equal to (c.f.  \eqref{o1}):
\begin{eqnarray*}
\adpor hBA&=&h\nor{\ad\xix}\adsta(B,A)+o(h).
\end{eqnarray*}

At the same time, by the definition of the reversed process, $\adpor hBA=\por hAB$ so
$\nor{\ad\xix}\adsta(B,A)=\nor{\xix}\sta(A,B)+o(1)$, implying first $\nor{\ad\xix}=\nor{\xix}$ and then 
$\adsta(B,A)=\sta(A,B)$.
\end{proof}

We can also write this as
\begin{eqnarray}
\adml y{dx}(\ad\la(y)\nu(dy)+\ad\sigma(dy))&=
\ml x{dy}(\la(x)\nu(dx)+\sigma(dx)).\label{ecce}
\end{eqnarray}
Integrating over $y\in A\subseteq\sE$, $x\in B\subseteq\ad\sE$ we obtain
\begin{align*}
\int_{y\in A}\adml y{B}(\ad\la(y)\nu(dy)+\ad\sigma(dy))&=
\int_{x\in B}\ml x{A}(\la(x)\nu(dx)+\sigma(dx)).
\end{align*}
The terms in the parentheses in \eqref{ecce} coincide with the stationary distributions of $\BEF$ and $\ad{\BEF}$ respectively. Hence we obtain the shorter representation
\begin{eqnarray}
\ml x{dy}\stbef(dx)
=
\adml y{dx}\ad\stbef(dy).\label{ecce2}
\end{eqnarray}
In particular the stationary distributions of the imbedded Markov chains correspond to each other: $\staft=\ad{\stbef}$ and $\stbef=\ad{\staft}$.

Frequently the jump measure $\mls x$ is absolutely continuous with respect to some other measure $\mu$ for every $x\in\sEA$. If this is the case, the same is true for the jump measure of the reversed process as the following corollary of Theorem \ref{main} shows.
\begin{Proposition}\label{absolu}
Suppose that $\mls x\abscon\mu$ with some measure $\mu$ on $\sE$ for every $x\in\sEA$. Then $\staft\abscon\mu$ and
\begin{align}
\radon{\staft}{\mu}(y)\adml y{dx}=\radon{\mls x}{\mu}(y)\stbef(dx).\label{eh}
\end{align}
The jump intensity of the reversed process fulfils
\begin{eqnarray}
\label{ka1}
\ad\la(y)\nu(dy)=\nor{\xix}\radon{\staft}{\mu}(y)\;\res{\mu}{\sEA}(dy).
\end{eqnarray}
For the boundary measure we have $\ad\sigma\abscon\res{\mu}{\acm}$ and
\begin{eqnarray}
\label{ka2}
\ad\sigma(dy)=\nor{\xix}\radon{\staft}{\mu}(y)\;\res{\mu}{\acm}(dy).
\end{eqnarray}
\end{Proposition}
\begin{remark}
If $\mu_x$ is absolutely continuous for every $x\in\sEA$ and $r$ is continuously differentiable
then by Proposition \ref{prop} the stationary measure $\nu$ is absolutely continuous, which simplifies most calculations.
\end{remark}

\begin{proof} It follows from   \eqref{piq} that the stationary distribution of $\AFT$ is absolutely continuous w.r.t. $\mu$, with  derivative given by:
\begin{align}
\radon{\staft}{\mu}(y)=\int_{x\in\sEA} \radon{\mls x}{\mu}(y)\stbef(dx)\label{nothing}.
\end{align}
By Theorem \ref{main}, equation \eqref{ecce2} and the fact that $\ad\stbef=\staft$ it follows that
\begin{align*}
\radon{\mls x}{\mu}(y)\mu(dy)\stbef(dx)
=\adml y{dx}\ad\stbef(dy)=
\adml y{dx}\staft(dy)=\radon{\staft}{\mu}(y)\adml y{dx}\mu(dy),
\end{align*}
implying \eqref{eh}. From \eqref{ecce} and \eqref{eh} it follows now that
\begin{align*}
\ad\la(y)\nu(dy)+\ad\sigma(dy)=\nor{\xix}\radon{\staft}{\mu}(y)\;\mu(dy),
\end{align*}
from which \eqref{ka1} and \eqref{ka2} can be deduced. 
\end{proof}

\begin{remark}
We have $\mls x\abscon \mu$ for example if  $\mls x(A)=\mu(A\cap B_x)/\mu(B_x)$ for some set $B_x$ with $\mu(B_x)>0$. Then the Radon-Nikodym derivative is given by
\begin{align*}
\radon{\mls x}{\mu}(y)=\frac{\ind{y\in B_x}}{\mu(B_x)}.
\end{align*}
For example, if $\mu$ is the Lebesgue-measure then $\mls x$ is the uniform distribution on $B_x$.
\end{remark}

\begin{remark}
Suppose that
$\AFT_i=g_{\BEF_i}(B_i)$, where $g_x:\sR^d\to\sE$ is  continuously  differentiable with $\det Dg_x\not=0$ (here $D$ is the Jacobian) and
 $B_1,B_2,\ldots$ are i.i.d. random variables with values in
$C=\{z\in\sR^d:g_x(z)\in\sE,\forall x\in\sE\}$ (we assume that $\sE$ is such that $C$ is non-empty) and an absolutely continuous distribution with density $f$. Then
\begin{align*}
\ml xA
&=\int_{C} \ind{g_x(y)\in A}f(y)\,dy
=\int_{g_x(C)} \frac{\ind{y\in A}f(g_x^{-1}(y))}{|\det Dg_x(y)|}\,dy
=\int_{A} \frac{\ind{y\in g_x(C)}f(g_x^{-1}(y))}{|\det Dg_x(y)|}\,dy,
\end{align*}
so that $\mls x\abscon \leb$ with density
\begin{align*}
\mls x'(y)=\ind{y\in g_x(C)}\frac{f(g_x^{-1}(y))}{|\det Dg_x(y)|}.
\end{align*}
Suppose for example that the frequent situation where $\AFT_i=\BEF_i+B_i$ is present, so that the jumps are random translations. Then $g_x(y)=x+y$, $C=\{z:z+x\in\sE,\forall x\in\sE\}$ and $\mls x$ is absolutely continuous with density
\begin{align*}
\mls x'(y)=\ind{y\in C+x}f(y-x).
\end{align*}
\end{remark}

According to  the classic literature on reversed Markov processes (see \cite{Naga,Nag2}) one  would expect that the generator $\ad{\gen}$ of $\ad X$ is the adjoint operator of $\gen$, that is
\begin{align*}
\int_{\sE}g(x)\gen f(x)\nu(dx)
 =\int_{\sEA}f(x)\ad\gen g(x)\nu(dx),
\end{align*}
for bounded functions $f,g$ in $\dom\gen\cap\dom{\ad\gen}$. This is indeed true and in the case of PDMPs this result takes the following form:

\begin{Theorem}
For $f,g\in\bacon$ we have
\begin{align}
&\int_{\sE}g(x)\gen f(x)\nu(dx)+\int_{\act}g(x)\qen f(x)\sigma(dx)
\nonumber\\&\quad=\int_{\sEA}f(x)\ad\gen g(x)\nu(dx)+\int_{\acm}f(x)\ad\qen g(x)\ad\sigma(dx).
\label{adju}
\end{align}
\end{Theorem}

\begin{proof}
Using the representation $\gen f(x)=\xen  f(x)+\la(x)\qen f(x)$, it follows that
\begin{eqnarray}
&&\int_{\sE}g(x)\gen f(x)\nu(dx)+\int_{\act}g(x)\qen f(x)\sigma(dx)
\nonumber\nline=\int_{\sE}g(x)\xen  f(x)\nu(dx)+\int_{\sEA}g(x)\qen f(x)\beck{\la(x)\nu(dx)+\sigma(dx)}.\label{ggx}
\end{eqnarray}
Leibnitz's rule $\xen  (fg)=\xen  f\cdot g+f\cdot \xen  g$ and the definition of $\qen$ imply that the r.h.s. of \eqref{ggx} is equal to
\begin{eqnarray}
&&\int_{\sE}\xen  (fg)(x)\nu(dx)-\int_{\sE}f(x)\xen  g(x)\nu(dx)
\nonumber\nline+\int_{\sEA}\int_{\sE}(g(x)f(y)-g(x)f(x))\ml{x}{dy}\beck{\la(x)\nu(dx)-\sigma(dx)}.\label{ggy}
\end{eqnarray}
It follows from \eqref{zero} that the integrals w.r.t. $\nu$ of $\xen  (fg)(x)$  and $-\la(x)\qen (fg)(x)$ coincide. Hence \eqref{ggy} equals:
\begin{eqnarray*}
&&-\int_{\sE}\la(x)\qen (fg)(x)\nu(dx)-\int_{\sE}f(x)\xen  g(x)\nu(dx)
\nline+\int_{\sEA}\int_{\sE}(g(x)f(y)-g(x)f(x))\ml{x}{dy}\beck{\la(x)\nu(dx)-\sigma(dx)},
\end{eqnarray*}
which, again by the definition of $\qen$, produces:
\begin{eqnarray*}
-\int_{\sE}f(x)\xen  g(x)\nu(dx)+\int_{\sEA}\int_{\sE}(g(x)f(y)-g(y)f(y))\ml{x}{dy}\beck{\la(x)\nu(dx)-\sigma(dx)}.
\end{eqnarray*}
Since $\ml{x}{dy}\beck{\la(x)\nu(dx)-\sigma(dx)}=\adml{y}{dx}\beck{\adla(y)\nu(dy)-\ad\sigma(dy)}$  by Theorem \ref{main}, this is equal to
\begin{align*}
&-\int_{\sEA}f(x)\xen  g(x)\nu(dx)+\int_{\sE}\int_{\sEA}(g(x)f(y)-g(y)f(y))\adml{y}{dx}\beck{\adla(y)\nu(dy)-\ad\sigma(dy)}
\\&\quad=\int_{\sEA}f(x)\ad\gen g(x)\nu(dx)+\int_{\acm}f(x)\ad\qen g(x)\ad\sigma(dx).\qedhere
\end{align*}
\end{proof}

\section{Processes on the real line}
\label{onedi}
\subsection{Introduction}
In this section we assume that $X_t$ has values in $\sR$ and that $t\mapsto \hh(x,t)$ is strictly increasing, that is $\xen  f(x)=r(x)f'(x)$ with some locally Lipschitz-continuous function $r:\sE\to\sR$ with $r(x)>0$, $x\in\sEE$. The state space of $X_t$ is a real interval $[w,\gamma)$, $w<\gamma$, where it is allowed to have either $w=-\infty$ (then $\sE=(-\infty,\gamma)$) or $\gamma=\infty$ (then $\sE=[w,\infty)$)  or both (in which case $\sE=\sR$). We assume that $\tact(x)<\infty$ for some (and then all) $x\in\sE$, whenever $\gamma<\infty$, so that $\{\gamma\}$ can be reached in finite time (otherwise one could transform the state space to obtain the $\gamma=\infty$ case). Hence, the active boundary is empty if $\gamma=\infty$ and is equal to $\{\gamma\}$ else.

The following figure shows  typical sample paths for the case where  $\la(x)=1$ ($N_t$ is a Poisson process), $r(x)=4-x$ i.e. $\hh(x,t)=4-(4-x)e^{-t}$, $\gamma=2$ (so $\act=\{2\}$ and $\sE=(-\infty,2)$) and $\mu_x(dy)=e^{-y} dy$, i.e.the process jumps from $x$ to $x-Z$, where $Z$ is exponential with mean one.
\fig{.97}{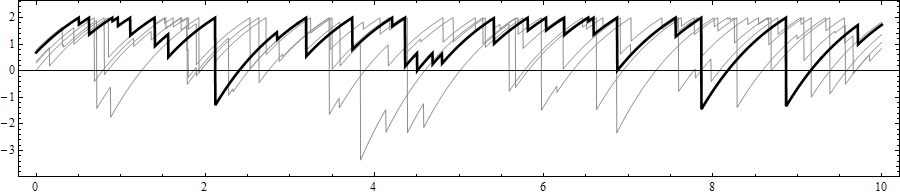}

Conditions for the existence of a stationary distribution $\nu$ can be found in \cite{Zheng, lastergo}.
The boundary measure $\sigma$ is determined by the single value $\sigma(\{\gamma\})$ which can be calculated as follows. According to \eqref{rice} we have
\begin{align*}
\sigma(\{\gamma\})&=\lim_{h\to 0}\frac{\nu(\{x:\tim x\gamma<h\})}{h}
=\lim_{h\to 0}\frac{1}{h}\int_{[w,\gamma]} \ind{\tim x\gamma<h} \nu(dx)
\\&=\lim_{h\to 0}\frac{1}{h}\int_0^h  r(\hh(\gamma,-s))\nu'(\hh(\gamma,-s))\,ds=r(\gamma)\nu'(\gamma).
\end{align*}
The following Theorem generalizes known formulas for the stationary density (see e.g. \cite[p.383]{Asmussen}, \cite[Theorem 2.1]{rene}, \cite[Theorem 1]{AL-ISO}, \cite[Proposition 4]{Bea}, c.f. \cite[Proposition 2.2]{lastrice}).

\begin{Theorem}\label{theo} The stationary distribution $\nu$ is absolutely continuous on $[w,\gamma)$. The density $\nu'$ of $\nu$ fulfills the integro-differential equation:
\begin{align}
r(x)\nu'(x)
&=\int_w^\gamma\la(z)\KK xz\nu'(z)\,dz+ \sigma(\{\gamma\})\KK x\gamma,\quad x\in(w,\gamma),
\label{qqq}
\end{align}
where $\KK xz=\ind{x\leq z}\ml z{[w,x)}-\ind{x>z}\ml z{[x,\gamma)}
=\ml z{[w,x)}-\ind{x>z}$.
\end{Theorem}
\begin{proof}
 Suppose that $f\in\bacon$. Then $f$ is absolutely continuous and $f(y)-f(x)=\int_x^y f'(u)\,du$, with a measurable $f':\sEE\to\sR$. Hence, using Fubini's theorem,
\begin{eqnarray*}
\qen f(z)
&=&\int_{w}^\gamma \kla{f(y)-f(z)}\,\ml z{dy}=\int_{w}^\gamma \int_z^y f'(x)\,dx\,\ml z{dy}
\\
&=&\int_{w}^\gamma \int_{w}^\gamma \ind{z<x\leq y}\,\ml z{dy} f'(x)\,dx
-\int_{w}^\gamma \int_{w}^\gamma \ind{y<x<z}\,\ml z{dy}  f'(x)\,dx
\\
&=&\int_{w}^\gamma\bkla{\ind{x>z} \ml z{[x,\gamma)}
-\ind{x\leq z} \ml z{[w,x)}} f'(x)\,dx.
\end{eqnarray*}
It follows from \eqref{zero}, applying Fubini once again, that
\begin{eqnarray}
&&\int_{[w,\gamma)} r(x)f'(x)\nu(dx)
\,=\,-\int_{[w,\gamma)} \la(z)\qen f(z)\nu(dz)
-\sigma(\{\gamma\})\qen f(\gamma).
\nonumber
\nline=\int_{w}^\gamma\Big(
\int_{[x,\gamma)} \la(z)  \ml z{[w,x)}\nu(dz)
\nonumber\nline\quad-\int_{[w,x)} \la(z) \ml z{[x,\gamma)}\nu(dz)
+\sigma(\{\gamma\})\ml \gamma{[w,x)}
\Big) f'(x)\,dx.\label{l}
\end{eqnarray}
Hence $\nu$ is absolutely continuous with density $\nu'$ fulfilling \eqref{qqq}.
\end{proof}

\subsection{The reversed process} Given a PDMP on the real line and a fixed time $T>0$, we define $\ad X_t=X_{(T-t)-}$ as before. As in the general case, under the \cd{DD} this reversed process is a PDMP.
Its deterministic path is decreasing with  $\ad r(x)=-r(x)$ and the active boundary is $\{w\}$ if $w>-\infty$. In contrast to the higher dimensional setting, in the present situation the integro-differential equation \eqref{qqq}  for $\nu$ enables us to express $\adla$ more explicitly.

\begin{remark}\label{condone dim}
In practical applications one may use the following sufficient conditions to ensure that the conditions \CD{AA}{AA1}--\CD{DD}{DD1} hold.
\begin{enumerate}
\item $\lambda$ is continuous, bounded and bounded away from zero. This condition implies \CD{AA}{AA1}, \CD{AA}{AA2}, \CD{BB}{BB2}, \CD{BB}{BB3} and (\ref{COC}).
\item For all $x\in\ad\sE$,  $\mu_x(B)$ is absolutely continuous with density $\mlm x{y}$, which is a
continuous function of $x$. This implies \CD{AA}{AA4} and \CD{BB}{BB4}.
\item there exists a stationary distribution $\nu$ (which is necessarily absolutely continuous, see Proposition \ref{prop}) with density fulfilling $\nu'(x)\not=0$ for $x\in\sEE$; then \CD{BB}{BB1} is obvious and since
\begin{align}
\frac{d\staft(z)}{d\nu(z)}= \frac{\int_{\ad\sE}  \mlm x{z}\la(x)\nu'(x)\,dx}{\nu'(z)}.\label{cat}
\end{align}
it follows that also \CD{DD}{DD1} is implied, too.
\item there is no active boundary or there exists $\epsilon >0$ such that $\ml x{\tau(x)>\epsilon}=1$ for all $x\in \act$, implying \CD{BB}{BB5} as well as \CD{AA}{AA3} (along with (1), see \cite[(26.4), p. 60]{Davis}).
\end{enumerate} 
Condition \CD{DD}{DD2}, local integrability of $\adla$, has to be checked in each particular case. If for example $\mu_x'$ is bounded and $1/\nu'(z)$ is locally integrable, then \CD{DD}{DD2} is fulfilled (see equation \eqref{cat}).
\end{remark}

\begin{Theorem}\label{solar} If $\mls x$ is absolutely continuous for all $x\in\sE$ with density $\mu_x^\prime$ and $r$ is absolutely continuous then so is $\nu'$ (with derivative $\nu''$) and if $\nu'(x)\not=0$, then
\begin{eqnarray}
\ad\la(x)&=&\la(x)+r'(x)+r(x)\frac{\nu''(x)}{\nu'(x)},\quad x\in\sEA,\label{kdo}
\\
\admlm x{y}&=&\frac{\la(y)\nu'(y)}{\adla(x)\nu'(x)}\mlm y{x},\quad x\in\sE,\; y\in\sEA,\label{kdu}
\\
\adml x{\{\gamma\}}&=&\frac{\sigma(\{\gamma\})}{\adla(x)\nu'(x)}\mlm \gamma{x},\quad x\in\sEA.\label{kuh}
\end{eqnarray}

\end{Theorem}
\begin{proof}
From Theorem \ref{theo} it follows that $\nu'$ is absolutely continuous
and by \eqref{ka1} and \eqref{piq} we have:
\begin{align}\label{lala}
\nu'(u)\bkla{\ad\la(u)-\la(u)}
&=\int_{\sE}\la(x)\mlm x{u}\nu'(x)\,dx+\int_{\act}\mlm x{u}\sigma(dx)-\la(u)\nu'(u).
\end{align}
Hence from \eqref{qqq} it follows that
\begin{align}
\nu'(u)\bkla{\ad\la(u)-\la(u)}
&=
\int_{[u,\gamma)}\la(x)\mlm x{u}\nu'(x)\,dx-\la(u)\nu'(u)\ml u{[u,\gamma)} \nonumber
\\&\quad +\int_{(-\infty,u)}\la(x)\mlm x{u}\nu'(x)\,dx-\la(u)\nu'(u)\ml u{(-\infty,u)}+\sigma(\{\gamma\})\mlm \gamma{u}
\nonumber\\&=(r(u)\nu'(u))'\nonumber
\end{align}
 which proves  \eqref{kdo}.
 The relation \eqref{kdu}
 follows from Proposition \ref{absolu} since
 \begin{align*}
\admlm x{y}=\frac{\pi_W^\prime(y)}{\pi_Q^\prime(x)}\mlm y{x}
=\frac{\pi_W^\prime(y)}{\pi_{W^\star}^\prime(x)}\mlm y{x}=\frac{\la(y)\nu'(y)}{\adla(x)\nu'(x)}\mlm y{x}.
 \end{align*}
 The proof of \eqref{kuh} is similar.\end{proof}

\begin{remark}
\label{remo}
It follows from \eqref{kdo} that
\begin{align*}
\ad\la(x)-\la(x)&=r(x)\bkla{\log(r(x)\nu'(x))}'
\end{align*}
Suppose that we observe the stationary process in state $x$. If $(\log(r(x)\nu'(x)))'$ is positive (negative) then  it is more likely (less likely)  to find a jump in the near past than to find a jump in the near future.

\end{remark}

Let assume that  the processes $X_t,\BEF_k,\AFT_k$ have been continued to the left. This can easily be achieved by setting $X_t$ equal to $\ad Y_{-t}$ for $t<0$, where $\ad Y$ is the time reversal of an independent version $\seq{Y}t$ of the PDMP.
Then $\AFT_{-1}$ denotes the state of $X_t$ just after the last jump before time $0$.

\begin{Corollary}
If $\mls x$ is absolutely continuous for all $x\in\sE$, then
\begin{eqnarray}
r(y)\nu'(y)\PX y{\BEF_1>x}=r(x)\nu'(x)\PX x{\AFT_{-1}<y},\quad y\leq x\in \sE.\label{piroschki}
\end{eqnarray}
\end{Corollary}

\begin{proof}
Dividing \eqref{kdo} by $\nu'(u)$ and integrating yields
\begin{eqnarray*}
\int_y^x\frac{\ad\la(u)-\la(u)}{r(u)}\,du=\int_y^x\frac{(r(u)\nu'(u))'}{r(u)\nu'(u)}\,du=\log\bkla{\frac{r(x)\nu'(x)}{r(y)\nu'(y)}}.
\end{eqnarray*}
Using $\ad r(x)=-r(x)$, we obtain:
\begin{eqnarray*}
r(y)\nu'(y) e^{-\int_y^x \frac{\la(u)}{r(u)}\,du}=r(x)\nu'(x)e^{-\int_x^y\frac{\ad\la(u)}{\ad r(u)}\,du}.
\end{eqnarray*}
By \eqref{u2} this is equivalent to
\begin{eqnarray*}
r(y)\nu'(y) e^{-\int_0^{\tim xy} \la(\hh(y,s))\,ds}=r(x)\nu'(x)e^{-\int_0^{\adtim xy}\adla(\hh(x,s))\,ds},
\end{eqnarray*}
which is equivalent to \eqref{piroschki}.
\end{proof}

We assume in the following examples that the conditions \CD{AA}{AA1}--\CD{BB}{BB4} and \cd{DD} are satisfied. See Remark \ref{condone dim} for sufficient conditions.

\subsection{Renewal age process}
\figc{.97}{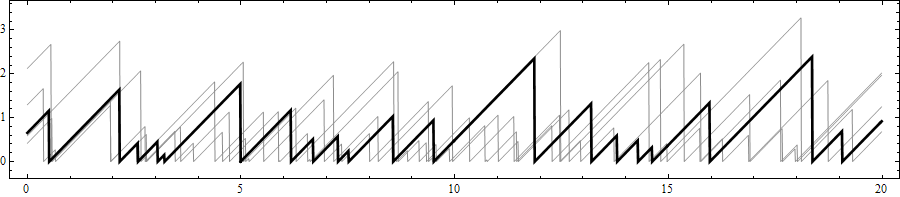}{Forward recurrence time $X_t$ with $F$ being a half-normal distribution with mean $1$.}
Consider the classic renewal process $N_t$ with renewal times $T_1,T_2,\ldots$. If the distribution of $T_1$ is absolutely continuous with distribution $F(x)$ and density $f(x)$ then the backward recurrence time $X_t=t-T_{N_t}$ is a PDMP with $r(x)=1$ and $\la(x)$ being equal to the hazard rate $f(x)/(1-F(x))$. The associated jump measure is, independently of $x$, concentrated in zero and the active boundary is empty (hence $w=0$ and $\gamma=\infty$).
It follows from \eqref{qqq} that, provided that $\EE{T_1}$ is finite, the density of the stationary distribution fulfills:
\begin{align*}
\nu'(x)
=\int_{[x,\infty)} \frac{f(z)}{1-F(z)}\nu'(z)\,dz.
\end{align*}
This leads to the well known fact that $\nu$ is equal to the equilibrium or stationary excess distribution of $F$:
\begin{align*}
\nu'(x)=\frac{1-F(x)}{\EE{T_1}}.
\end{align*}
The reversed process $\ad X_t$ is the forward recurrence time process of the same renewal process. Applying Theorem \ref{main} we see that  $\ad\la(y)=0$ for $y\in\sEA$ and that
\begin{align*}
f(x)\,dx
=\adml 0{dx}\ad\sigma(\{0\})
\end{align*}
from which, upon integration, $\ad\sigma(\{0\})=1/\EE{T_1}$ follows (the average number of visits to $0$) and hence $\adml 0{dx}=f(x)\,dx$, which is not surprising: the upward jumps of the reversed process are just i.i.d. jumps with distribution $F$. Hence, for the trivial case of the renewal process our formulas yield the correct results.

\subsection{Generalized TCP window size process}\label{gw}

\figc{.97}{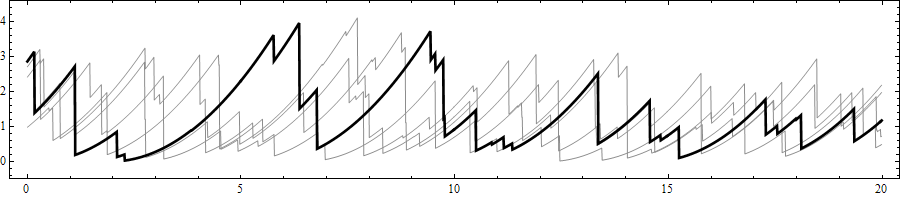}{Generalized TCP window size process with $\alpha=1/2$, $\beta=1$, $\gamma=1$.}

In \cite{AL-ISO} a generalized model for the TCP window size has been studied (see also  \cite{altmanmap, dumas, chafaii, AL-TRANS,bkim, bert1}). We assume that $X_t$ is a process with  $\la(x)=\la x^\beta$, $r(x)=rx^\alpha$, $\beta>\alpha-1$ and $\mu_x((0,y])=(y/x)^\gamma$, which means that $\AFT_{k}=U_K^{1/\gamma}\cdot \BEF_k$, where the $U_k$ are independent random variables with uniform distribution on $(0,1)$. In other words, at jump times the process is multiplied by the random number $U_k^{1/\gamma}$.

Using Theorem \ref{theo} it can be shown  that under the condition $\beta>\alpha-1$ a unique stationary distribution $\nu$ exists and
\begin{eqnarray*}
\nu'(x)=Cx^{\gamma-\alpha}\exp\bkla{-\frac{\la}{r(\beta-\alpha+1)}x^{\beta-\alpha+1}}
\end{eqnarray*}
(see  (27) in \cite{AL-ISO} for the normalizing constant $C$).
Surprisingly, it follows from
\begin{eqnarray*}
\frac{\nu''(x)}{\nu'(x)}=\frac{\gamma-\alpha}{x}-\frac{\la x^{\beta-\alpha}}{r}
\end{eqnarray*}
that the jump intensity of the reversed process is of the same form as $\la$:
\begin{eqnarray*}
\ad\la(x)=\la(x)+r'(x)+r(x)\frac{\nu''(x)}{\nu'(x)}=r\gamma x^{\alpha-1}.
\end{eqnarray*}
As long as $\beta>0$ the reversed process $\ad X$ cannot reach $0$ in finite time. If $\beta\leq 0$ then the process could reach $0$, but since $\alpha-1<\beta\leq -1$, it follows that $\adla$ is not integrable on $[0,\veps)$ for any $\veps$: the process will almost surely (and by construction even surely) jump before it reaches $0$.

The jump measure of the reversed process is of a different form than the one for the original process, namely
\begin{eqnarray}
\admlm x{y}=\frac{\la(y)\mlm y{x}\nu'(y)}{\adla(x)\nu'(x)}
=\frac{ \la}{r}y^{\beta-\alpha} \exp\Bkla{\frac{\la}{r(\beta-\alpha+1)}\bkla{x^{\beta-\alpha+1}-y^{\beta-\alpha+1}}},
\end{eqnarray}
for $y\geq x$, so that we obtain a Weibull-type distribution
\begin{eqnarray}
\adml x{(x,y]}=\exp\Bkla{\frac{\la}{r(\beta-\alpha+1)}\bkla{x^{\beta-\alpha+1}-y^{\beta-\alpha+1}}}.
\end{eqnarray}
If $\alpha=\beta$ then the upward jumps of $\ad X$ follow an exponential distribution with mean $r/\la$ and the reversed process has additive jumps whereas the original process has multiplicative jumps (this has been observed in \cite{gideon}).

\subsection{Jumps independent of the current state}
\figc{.97}{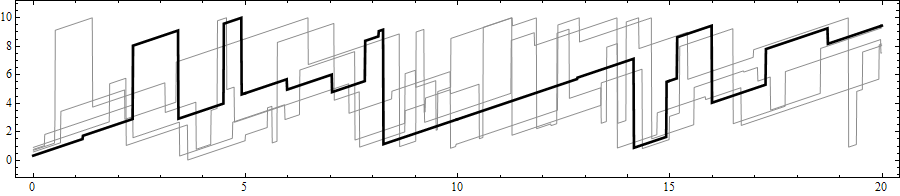}{$r(x)=1$, $\la(x)=1$, $\mu_x(A)=\leb(A\cap [0,10])$, $\gamma=10$.}
Suppose that $\ml z{A}=\mu(A)$ for all $z\in\sE$ and some absolutely continuous  measure $\mu$ on $[w,\gamma)$. Then it follows from \eqref{qqq} that
\begin{align*}
r(x)\nu'(x)
&=\int_w^\gamma\la(z)\KK xz\nu'(z)\,dz+ \sigma(\{\gamma\})\KK x\gamma,\quad x\in(w,\gamma),
\end{align*}
where $\KK xz=\mu([w,x))-\ind{x>z}$, implying that
\begin{align*}
r(x)\nu'(x)
&=\nor\xix\mu([w,x))-\int_w^x\la(z)\nu'(z)\,dz,
\end{align*}
where $\nor\xix=\int_w^\gamma\la(z)\nu'(z)\,dz+ \sigma(\{\gamma\})$. It follows that
\begin{align}\label{lloef}
r(x)\nu''(x)+\kla{r'(x)+\la(x)}\nu'(x)
&=\nor\xix\mu'(x)
\end{align}
and hence, using $\nu'(w)=0$,
\begin{align*}
\nu'(x)=\nor\xix\frac{1}{r(x)}
\int_w^x e^{-\int_u^x \frac{\la(y)}{r(y)}\,dy}\,d\mu(u).
\end{align*}
Then it follows from \eqref{kdo} and \eqref{lloef} that
\begin{align*}
\adla(x)=\nor\xix\frac{\mu'(x)}{\nu'(x)},
\end{align*}
which is nothing else than \eqref{ka1} since $\staft=\mu$ here. Note that $\nor\xix$ can be found from the normalizing condition $\nu(\sE)=1$.

\subsection{Reflection}

Suppose that $X_t$ is a real valued PDMP with decreasing paths with the additional feature that once the process reaches $0$ the process stays in $0$ for an exponentially distributed time with mean $1/\la(0)$. We assume that $r(x)<0$ for all $x>0$, so $0$ can actually be reached. Moreover the time to reach zero has finite mean. Technically this process is not a PDMP on the real line. Instead the process can be modelled as a PDMP with two outer states. We let $\sEE_1=(0,\varsigma]$ with active boundary $\act_1=\{0\}$ and  introduce a second component $\sEE_2$ which includes the point $\{0\}$. Also let $r(x)=0$ for all $x\in\sEE_2$. Then the state space of $X_t$ is $\{(x,i)|x\in\sEE_i,i\in{1,2}\}$. Once $X_t$ hits $(0,1)$ it jumps to $(0,2)$ and stays there until it jumps back to $\sEE_1$. However, we avoid this cumbersome notation and just write $\sE=[0,\varsigma]$ and $\act=\{0\}$ and thereby allow $\sE\cap\act\not=\emptyset$. Note that the reversed process $\ad X_t$, after jumping to $0$, stays there for an exponentially distributed sojourn time until it starts decreasing again.

\begin{Theorem}\label{theo2} The stationary distribution $\nu$ is absolutely continuous on the interval $(0,w)$. The density $\nu'$ of $\nu$ fulfills the integro-differential equation:
\begin{align}
r(x)\nu'(x)
&=\int_0^\varsigma \la(z)\MM xz\nu'(z)\,dz+ \la(0)\nu(\{0\})\MM x0,\quad x\in(0,\varsigma),
\label{qqq2}
\end{align}
where $\MM xz=\ind{x\leq z} \ml z{[0,x)}-\ind{x>z} \ml z{[x,\varsigma)}
=\ml z{[0,x)}-\ind{x>z}$.
\end{Theorem}

\begin{proof} As in the proof of Theorem \ref{theo} let
 $f\in\bacon$. By Fubini's theorem
\begin{eqnarray*}
\qen f(z)
&=&\int_{0}^\varsigma \kla{f(y)-f(z)}\,\ml z{dy}=\int_{0}^\varsigma \int_z^y f'(x)\,dx\,\ml z{dy}
\\
&=&\int_{0}^\varsigma \int_{0}^\varsigma \ind{z<x\leq y}\,\ml z{dy} f'(x)\,dx
-\int_{0}^\varsigma \int_{0}^\varsigma \ind{y<x<z}\,\ml z{dy}  f'(x)\,dx
\\
&=&\int_{0}^\varsigma\bkla{\ind{x>z} \ml z{[x,\varsigma)}
-\ind{x\leq z} \ml z{[0,x)}} f'(x)\,dx.
\end{eqnarray*}
It follows from \eqref{zero} and $\qen f(\gamma)=0$ that
\begin{align*}
\int_{(0,\varsigma)} r(x)f'(x)\nu(dx)
&=-\int_{(0,\varsigma)} \la(z)\qen f(z)\nu(dz)
\\&=-\int_{(0,\varsigma)} \Bkla{\int_{0}^\varsigma\la(z)\MM xz \nu(dz)+\la(0)\MM x0\nu(\{0\})}f'(x)\,\,dx,
\end{align*}
so $\nu$ is absolutely continuous and  \eqref{qqq2} holds.
\end{proof}

The parameters of the reversed process are given in Theorem \ref{solar} applied for $-X_t$ and $\gamma=0$, $w=-\varsigma$.

\subsection{Workload of the single server queue M/G/1} \label{risk}
\figc{.97}{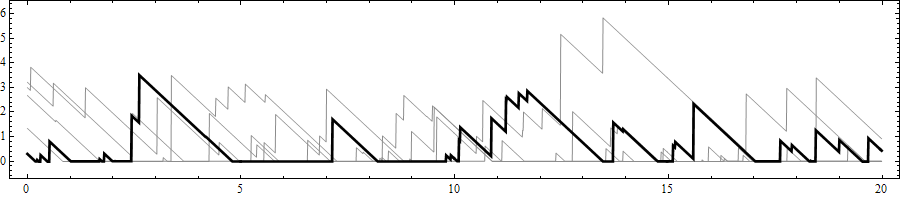}{Risk process}
The workload  (also virtual waiting time) process of a Markovian single server queue is defined as the amount of work left at the server at time $t$.
This is a special case of the reflected process with $\varsigma=\infty$, where $r(x)=-1$, $\la(x)=\la$ is constant and $\ml z{(x,\infty]}=1-F(x-z)$ for $z\leq x$, where $F$ is a probability distribution on $(0,\infty)$. We assume that $\varrho=\la\int_0^\infty (1-F(u))\,du<1$, so that it can be shown that unique stationary distribution $\nu$ exists. It then follows from \eqref{qqq2} that
\begin{eqnarray}
\nu'(x)
=\la\int_{0}^x (1-F(x-z))\nu'(z)\,dz+ \la\nu(\{0\})(1-F(x)).\label{polkin}
\end{eqnarray}
Integrating yields
$
1-\nu(\{0\})
=\varrho(1-\nu(\{0\}))+ \nu(\{0\})\varrho
$
with $\varrho=\la/\int_0^\infty(1-F(u))\,du$, from which $\nu(\{0\})=1-\varrho$ follows.

The renewal type equation \eqref{polkin} can be solved at least for numerical purposes in form of an infinite series of convolutions (the Pollaczek-Khintchine formula in the queueing context). Alternatively one can take Laplace transforms and find an equivalent relation for the transform.

Provided that $F$ is absolutely continuous with density function $f$, Theorem \ref{solar} implies:
\begin{eqnarray*}
\ad\la(x)&=&\la+\frac{\nu''(x)}{\nu'(x)},\quad x>0,\label{pus}
\\
\admlm x{y}&=&\frac{\la  \nu'(y)}{\adla(x)\nu'(x)}f(x-y),\quad x \geq 0,\;y<x,
\nonumber\\
\adml x{\{0\}}&=&\frac{\la (1-\varrho)}{\adla(x)\nu'(x)}f(x),\quad x>0.\nonumber
\end{eqnarray*}

\begin{remark} It is interesting to look for conditions to obtain a constant jump rate for the reversed process. We have $\ad\la=\la+C$ for some constant $C$, iff $\nu''(x)/\nu'(x)=C$, that is, if $\nu'(x)=De^{Cx}$ is an exponential density  with additional mass at zero. Then
\begin{align*}
\admlm x{y}=\frac{\la  \nu'(y)}{\la\nu'(x)+\nu''(x)}f(y-x)
=\frac{\la }{\la+C}e^{C(y-x)}f(y-x)
\end{align*}
depends only on $(y-x)$. If we interpret $g(y)=\frac{\la }{(\la+C)}e^{Cy}f(y)$ as the density of the workload of an M/G/1 queue, then it follows from the fact that $\nu$ is a mixture of an exponential distribution with a mass at zero (see e.g. \cite{Asmussen}, Theorem 9.1) that $g(x)$ is also exponential with mean $1/\beta$ and $\beta=\ad\la-C=\la$.
\end{remark}

\begin{remark}
Specializing Remark \ref{remo}, if we ask ourself whether it is more probable, if we observe $X_t$ in steady state, to have a claim in the near future than to have a claim in the near past, \eqref{pus} shows that
 $\adla(x)>\la$ if $\log \nu'$ is increasing  at $x$ and
 $\adla(x)<\la$ if $\log \nu'$ is decreasing at $x$ (no matter which claim size distribution is present).
\end{remark}


\bibliographystyle{plain}

\end{document}